\documentclass[12pt]{article}
\usepackage{amsmath,amssymb,amsthm}
\newtheorem{theorem}{Theorem}[section]

\newtheorem{lemma}[theorem]{Lemma}

\newtheorem{definition}[theorem]{Definition}

\numberwithin{equation}{section}
\DeclareMathOperator{\rank}{rank}
\DeclareMathOperator{\spn}{span}
\def\F{\mathbb{F}_q}
\def\P{\mathcal{P}}
\def\L{\mathcal{L}}

\def\K{\mathcal{K}}

\def\HH{\mathcal{H}}

\begin{document}
\title{Incidences between planes over finite fields}

\author{
Nguyen Duy Phuong\thanks{University of Science, Vietnam National University Hanoi
    Email: {\tt duyphuong@vnu.edu.vn
}}
\and 
    Pham Van Thang\thanks{EPFL, Lausanne, Switzerland. Research partially supported by Swiss National Science Foundation
Grants 200020-144531 and 200021-137574.
    Email: {\tt thang.pham@epfl.ch}}
  \and
    Le Anh Vinh\thanks{University of Education, Vietnam National University Hanoi. Research was supported by Vietnam National Foundation for Science and Technology Development grant 101.99-2013.21.
    Email: {\tt vinhla@vnu.edu.vn
}}}
\date{}
\maketitle
\begin{abstract}
We use methods from  spectral graph theory to obtain bounds on the number of incidences between $k$-planes and $h$-planes in $\mathbb{F}_q^d$ which generalize a recent result given by Bennett, Iosevich, and Pakianathan (2014). More precisely, we prove that the number of incidences between a set $\mathcal{P}$ of $k$-planes and a set $\mathcal{H}$ of $h$-planes with $h\ge 2k+1$,  which is denoted by $I(\mathcal{P},\mathcal{H})$, satisfies
\[\left\vert I(\mathcal{P},\mathcal{H})-\frac{|\mathcal{P}||\mathcal{H}|}{q^{(d-h)(k+1)}}\right\vert \lesssim q^{\frac{(d-h)h+k(2h-d-k+1)}{2}}\sqrt{|\mathcal{P}||\mathcal{H}|}. \]
\end{abstract}
\section{Introduction}

Let $\F$ be a finite field of $q$ elements where $q$ is an odd prime power. Let $\P$ be a set of points, $\mathcal{L}$ a set of lines, and $I(\P,\L)$ the number of incidences between $\P$ and $\L$.  In \cite{bourgain-katz-tao} Bourgain, Katz, and Tao proved that the number of incidences between a point set of $N$ points and a line set $N$ lines is at most $O(N^{3/2-\epsilon})$. Here and throughout, $X\gtrsim Y$ means that $X\ge CY$ for some constant $C$ and $X\gg Y$ means that $Y=o(X)$, where $X,Y$ are viewed as functions of the parameter $q$. 

Note that one can easily obtain the bound $N^{3/2}$ by using the Tur\'{a}n number and the fact that two lines intersect in at most one point.  The relationship between $\epsilon$ and $\alpha$ in the result of Bourgain, Katz, and Tao is difficult to determine, and it is far from tight. If $N\ll q$, then Grosu \cite{gro} proved that one can embed the point set and the line set to $\mathbb{C}^2$ without changing the incidence structure. Thus it follows from a tight bound on the number of incidences between points and lines in $\mathbb{C}^2$ due to T\'{o}th \cite{to} that $I(\P,\L)=O(N^{4/3})$. By using methods from spectral graph theory, the third listed author \cite{vinh-incidence} gave a tight bound for the case $N>q$ as follows.

\begin{theorem}\label{vinh-pl}
Let $\P$ be a set of points and $\L$ be a set of lines in $\mathbb{F}_q^2$. Then we have 
\begin{equation}\label{vinh-in}\left\vert I(\P,\L)-\frac{|\P||\L|}{q}\right\vert\le q^{1/2}\sqrt{|\P||\L|}.\end{equation}
\end{theorem}

It follows from Theorem \ref{vinh-pl} that if $N\ge q^{3/2}$, then the number of incidences between points and lines is at most $(1+o(1))N^{4/3}$, which meets the Szemer\'{e}di-Trotter bound. Theorem \ref{vinh-pl} has many interesting applications in several combinatorial problems, see for example \cite{ hen1, hen2, area,vinh-incidence}.

It also follows from the lower bound that if $|\P||\L|\gtrsim q^3$, then there exists at least one pair $(p,l)\in \P\times \L$ such that $p\in l$. The lower bound of Theorem \ref{vinh-pl} is proved to be sharp up to a constant in the sense that there exist a point set $\P$ and a line set $\L$ with $|\P|=|\L|=q^{3/2}$ without incidences (see \cite{vinhpointline} for more details). Furthermore, the third listed author proved that almost every point set $\P$ and line set $\L$ in $\mathbb{F}_q^d$ of cardinality $|\P|=|\L|\gtrsim q$, there exists at least one incidence $(p,l)\in \P\times \L$. More precisely, the statement is as follows.
\begin{theorem}[\textbf{Vinh} \cite{vinhpointline}]
For any $\alpha>0$, there exist an integer $q_0=q(\alpha)$ and a number $C_\alpha>0$ satisfying the following property. When a point set $\P$ and a line set $\L$ with $|\P|=|\L|=s\ge C_\alpha q$, are chosen randomly in $\F^2$, then the probability of $\{(p,l)\in \P\times \L\colon p\in l \}\equiv\emptyset$ is at most $\alpha^s$, provided that $q\ge q_0$.
\end{theorem}

Using the same ideas, the third listed author \cite{vinh-incidence} generalized Theorem \ref{vinh-pl} to the case of points and hyperplanes in $\F^d$ as follows.
\begin{theorem}[\textbf{Vinh }\cite{vinh-incidence}]\label{plane}
Let $\P$ be a set of points and $\HH$ be a set of hyperplanes in $\F^d$. Then the number of incidences between points and hyperplanes satisfies
\[\left\vert I(\P,\HH)-\frac{|\P||\HH|}{q}\right\vert\le (1+o(1))q^{(d-1)/2}\sqrt{|\P||\HH|}.\]
\end{theorem}
By using counting arguments and the upper bound on the number of incidences between points and hyperplanes, Bennett, Iosevich, and Pakianathan \cite{iom} extended Theorem \ref{plane} to the incidences between points and $k$-planes, where a $k$-plane is defined as follows.
\begin{definition}
Let $V$ be a subset in the vector space $\F^d$. Then $V$ is a $k$-plane in $\F^d$, $k<d$, if there exist $k+1$ vectors $v_1,\ldots,v_{k+1}$ in $\F^d$  satisfying \[V=\spn\{v_1,\ldots,v_k\}+v_{k+1}, ~\rank\{v_1,\ldots,v_k\}=k.\]
\end{definition}

\begin{theorem}[\textbf{Bennett et al. }\cite{iom}]\label{io}
Let $\P$ be a set of points and $\HH$ be a set of $k$-planes in $\F^d$. Then there is no more than \[\frac{|\P||\HH|}{q^{d-k}}+(1+o(1))q^{k(d-k )/2}\sqrt{|\P||\HH|}\] incidences between the point set $\P$ and the plane set $\HH$.
\end{theorem}

In this paper, we will extend Theorem \ref{io} to the case of $k$-planes and $h$-planes with $h\ge 2k+1$ in the following theorem.
\begin{theorem}\label{thm1}
Let $\P$ be a set of $k$-planes and $\HH$ be a set of $h$-planes in $\mathbb{F}_q^d$ $(h\ge 2k+1)$. Then the number of incidences between $\P$ and $\HH$ satisfies
\[\left\vert I(\P,\HH)-\frac{|\P||\HH|}{q^{(d-h)(k+1)}}\right\vert\le \sqrt{2k+1}q^{\frac{(d-h)h+k(2h-d-k+1)}{2}}\sqrt{|\P||\HH|}. \]
\end{theorem}
It follows from Theorem \ref{thm1} that if $|\P||\HH|\gtrsim q^{d(k+h)+2d+k}/q^{k^2+h^2+2h}$ then the set of incidences between $\P$ and $\HH$ is nonempty, and if $|\P||\HH|\gg q^{d(k+h)+2d+k}/q^{k^2+h^2+2h}$ then  $I(\P,\HH)$ is close to the expected number $|\P||\HH|/q^{(d-h)(k+1)}$. 

The study of incidence problems over finite fields received a considerable amount of attention in recent years, see for example \cite{ covert,  iosevichetal, gro,hr, jo, kollar, thang, solymosi, lund,tv}. 

A related question that has recently received attention is the following: Given a point set $\P$ in $\mathbb{F}_q^2$, what is the cardinality of the set of $k$-rich lines, i.e. lines contain at least $k$ points from $\P$? Note that this question is quite different from the real case.  In the real case, it follows from the Szemer\'{e}di-Trotter theorem that the number of $k$-rich lines determined by a set of $n$ points is $O(n^2/k^3)$ for any $k\ge 2$, but in the finite fields case, it follows from Theorem \ref{vinh-pl} that the number of $k$-rich lines determined by a point set $\P$ is at most $q|\P|/(k-q^{-1}|\P|)^2$ with $k>|\P|/q$. In \cite{lund}, Lund and Saraf introduced an approach to  deal with this problem. More precisely, they proved that, for any $k\ge 2$, there are at least $cq^2$ $k$-rich lines determined by a point set of cardinality $2(k-1)q$ for some constant $0<c<1$. This implies that there are at least $cq^2$ distinct lines determined by a set of $2q$ points. They also proved that 
\begin{theorem}[\textbf{Lund-Saraf} \cite{lund}]\label{lusa}
For any integer $t\ge 2$, let $\HH$ be a set of the $h$-planes in $\F^d$ of the cardinality
\[|\HH|\ge 2(t-1)q^{d-h}.\]
Then the number of points contained in at least $t$ $h$-planes from $\HH$ is at least $cq^d$, where $c=(t-1)/(t-1+2q^{h(d-h-1)})$.
\end{theorem}
We note that in the case $h<d-1$ and $t<q^{h(d-h-1)}$, the constant $c$ depends on $q$. This condition is necessary since, for instance, one can take a set of $2q^2$ lines in the union of two planes in $\F^3$, then the number of $2$-rich points is at most $O(q^2)$. On the other hand, it follows from Theorem \ref{lusa} for the case $d=3$ and $h=1$ that the number of $2$-rich points is at least $\Omega(q^2)$. This implies that the theorem is tight in this case. 

Using Lund and Saraf's approach and the properties of plane-incidence graphs in Section 3, we obtain generalizations of their results as follows. 

\begin{theorem}\label{thm2}
For any $t\ge 2$, let $\HH$ be a set of $h$-planes in vector space over $\mathbb{F}_q^d$ with the cardinality \[ |\HH|\ge 2(t-1)q^{(d-h)(k+1)}.\] 
Then the number of $k$-planes contained in at least $t$ $h$-planes in $\HH$ is at least $cq^{(d-k)(k+1)}$, where $~c=(t-1)/\left((t-1)+2q^{(d-h-1)(h-k)+k}\right)$.
\end{theorem}
\begin{theorem}\label{thm3}
For any $t\ge 2$, let $\K$ be a set of $k$-planes in vector space over $\mathbb{F}_q^d$ with the cardinality \[|\K|\ge 2(t-1)q^{(d-h)(k+1)}.\] Then  the number of $h$-planes containing  at least $t$ $k$-planes in $\K$ is at least $cq^{(d-h)(h+1)}$, where $~c=(t-1)/\left((t-1)+2q^{k(h-k+1)}\right)$.
\end{theorem}
\section{Expander Mixing Lemma}

We say that a bipartite graph is \emph{biregular} if in both of its two parts, all vertices have the same degree. If $A$ is one of the two parts of a bipartite graph, we write $\deg(A)$ for the common degree of the vertices in $A$.
Label the eigenvalues so that $|\lambda_1|\geq |\lambda_2|\geq \cdots \geq |\lambda_n|$.
Note that in a bipartite graph, we have $\lambda_2 = -\lambda_1$. The following version of the expander mixing lemma is proved in \cite{eustis}. We give a detailed proof for the sake of completeness.
\begin{lemma}\label{expander} Suppose $G$ is a bipartite graph with parts $A,B$ such that the vertices in $A$ all have degree $a$ and the vertices in $B$ all have degree $b$. For any two sets $X\subset A$ and $Y\subset B$, the number of edges between $X$ and $Y$, denoted by $e(X,Y)$, satisfies
\[\left\vert e(X,Y)-\frac{a}{|B|}|X||Y|\right\vert\le \lambda_3\sqrt{|X||Y|},\] where $\lambda_3$ is the third eigenvalue of $G$.
\end{lemma}
\begin{proof}
We assume that the vertices of $G$ are labeled from $1$ to $|A|+|B|$. Let $M$ be the adjacency matrix of $G$ having the form
$$
M=\left(\begin{matrix}
0&N\\
N^t&0
\end{matrix}\right),
$$ where $N$ is the $|A|\times|B|$ matrix, and $N_{ij}=1$ if and only if there is an edge between $i$ and $j$. Firstly, let us recall some properties of eigenvalues of the matrix $M$. Since all vertices in $A$ have degree $a$ and all vertices in $B$ have degree $b$, all eigenvalues of $M$ are bounds by $\sqrt{ab}$. Indeed, let us denote the $L_1$ vector norm by $||\cdot||_1$, and $\mathbf{e}_v$ an unit vector having a $1$ in the position for vertex $v$ and zeroes elsewhere. It is easy to see that $||M^2\cdot \mathbf{e}_v||_1\le ab$. Therefore, all eigenvalues of $M$ are bounded by $\sqrt{ab}$. Let $\textbf{1}_X$ denote the column vector having $1\textrm{s}$ in the positions corresponding to the set of vertices $X$ and $0\textrm{s}$ elsewhere. Then we have
\[M(\sqrt{a}\mathbf{1}_A+\sqrt{b}\mathbf{1}_B)=b\sqrt{a}\mathbf{1}_B+a\sqrt{b}\mathbf{1}_A=\sqrt{ab}(\sqrt{a}\mathbf{1}_A+\sqrt{b}\mathbf{1}_B),\]
\[M(\sqrt{a}\mathbf{1}_A-\sqrt{b}\mathbf{1}_B)=b\sqrt{a}\mathbf{1}_B-a\sqrt{b}\mathbf{1}_A=-\sqrt{ab}(\sqrt{a}\mathbf{1}_A-\sqrt{b}\mathbf{1}_B),\]
which implies that $\lambda_1=\sqrt{ab}$ and $\lambda_2=-\sqrt{ab}$ are the first and the second eigenvalues corresponding to eigenvectors $(\sqrt{a}\mathbf{1}_A+\sqrt{b}\mathbf{1}_B)$ and $(\sqrt{a}\mathbf{1}_A-\sqrt{b}\mathbf{1}_B)$, respectively.

Let $W^\perp$ be a subspace spanned by two vectors $\textbf{1}_A$ and $\textbf{1}_B$. Since $M$ is a symmetric matrix, the eigenvectors of $M$ except $\sqrt{a}\textbf{1}_A+\sqrt{b}\textbf{1}_B$ and $\sqrt{a}\textbf{1}_A-\sqrt{b}\textbf{1}_B$ span $W$. Therefore,  for any $u\in W$, we have $Mu\in W$, and $||Mu||\le \lambda_3||u||$. We have the following observations.
\begin{itemize}
\item[1.] Let $K$ be a matrix of the form $\left(\begin{matrix}
0&J\\
J&0
\end{matrix}\right),$ where $J$ is the $|A|\times |B|$ all-ones matrix. If $u\in W$, then $Ku=0$ since every row of $K$ is either $\textbf{1}_A^T$ or $\textbf{1}_B^T$.
\item[2.] If $w\in W^\perp$, then $(M-(a/|B|)K)w=0.$ Indeed, it follows from the facts that $a|A|=b|B|$, and $M\textbf{1}_A=b\textbf{1}_B=(a/|B|)K\textbf{1}_A$, $M\textbf{1}_B=a\textbf{1}_A=(a/|B|)K\textbf{1}_B$.
\end{itemize}

Since $e(X,Y)=\textbf{1}_Y^TM\textbf{1}_X$ and $|X||Y|=\textbf{1}_Y^TK\textbf{1}_X$, \[\left|e(X,Y)-\frac{a}{|B|}|X||Y|\right|=\left|\textbf{1}_Y^T(M-\frac{a}{|B|}K)\textbf{1}_X\right|.\]
For any vector $v$, let $\bar{v}$ denote the orthogonal projection onto $W$, so that $\overline{v}\in W$, and $v-\overline{v}\in W^\perp$. Thus \[\textbf{1}_Y^T(M-\frac{a}{|B|}K)\textbf{1}_X=\textbf{1}_Y^T(M-\frac{a}{|B|}K)\overline{\textbf{1}_X}=\textbf{1}_Y^TM\overline{\textbf{1}_X}=\overline{\textbf{1}_Y}^T M \overline{\textbf{1}_X}.\] Therefore,
\[\left|e(X,Y)-\frac{a}{|B|}|X||Y|\right|\le \lambda_3||\overline{\textbf{1}_X}||~||\overline{\textbf{1}_Y}||.\]

Since \[ \overline{\textbf{1}_X}=\textbf{1}_X-\left((\textbf{1}_X\cdot \textbf{1}_A)/(\textbf{1}_A\cdot \textbf{1}_A)\right)\textbf{1}_A=\textbf{1}_X-(|X|/|A|)\textbf{1}_A,\] we have $||\overline{\textbf{1}_X}||=\sqrt{|X|(1-|X|/|A|)}$. Similarly, $||\overline{\textbf{1}_Y}||=\sqrt{|Y|(1-|Y|/|B|)}$.

In other words,
\[\left|e(X,Y)-\frac{a}{|B|}|X||Y|\right|\le \lambda_3\sqrt{|X||Y|(1-|X|/|A|)(1-|Y|/|B|)},\]
and the lemma follows.
\end{proof}
\begin{lemma}\label{lm1}
Let $G$ be a biregular  graph with parts $A,B$ and $|A|=m$, $|B|=n$. We label vertices of $G$ from $1$ to $|A|+|B|$. Let $M$ be the adjacency matrix of $G$ having the form
$$
M=\left(\begin{matrix}
0&N\\
N^t&0
\end{matrix}\right),
$$ where $N$ is the $|A|\times|B|$ matrix, and $N_{ij}=1$ if and only if there is an edge between $i$ and $j$. Let $v_3=(v_1,\ldots,v_m, u_1,\ldots, u_n)$ be an eigenvector of $M$ corresponding to the eigenvalue $\lambda_3$. Then we have $(v_1,\ldots,v_m)$ is an eigvenvector of $NN^T$, and $J(v_1,\ldots,v_m)=0$, where $J$ is the $m\times m$ all-ones matrix.
\end{lemma}
\begin{proof}
We have 
\[
M^2=\left(\begin{matrix}
NN^T&0\\
0&N^TN
\end{matrix}\right).\]
Since $v_3$ is an eigenvector of $M$ with eigenvalue $\lambda_3$, $v_3$ is also an eigenvector of $M^2$ with the eigenvalue $\lambda_3^2$. On the other hand,

\[
M^2v_3=\left(\begin{matrix}
NN^T\cdot(v_1,\ldots,v_m)^T\\
N^TN\cdot (u_1,\ldots,u_n)^T
\end{matrix}\right)=\lambda_3^2(v_1,\ldots,v_m,\\u_1,\ldots,u_n)^T.\]
This implies that $(v_1,\ldots,v_m)$ is an eigenvector of $NN^T$ corresponding to the eigenvalue $\lambda_3^2$. We also note that it follows from proof of Lemma \ref{expander} that if $v_3$ is an eigenvector  corresponding the third eigenvalue of $M$, then $Kv_3=0$, which implies that $J(v_1,\ldots,v_m)=0$. We also note that $\lambda_3^2$ is the second eigenvalue of $NN^T$.
\end{proof}
It follows from Lemma \ref{lm1} that in order to bound the third eigenvalue of $M^2$, it suffices to bound the second eigenvalue of the matrix $NN^T$.

Suppose that $G=(A,B,E)$ is a bipartite graph as in Lemma \ref{expander}. For any set $S$ of vertices in $A$, we denote  the set of vertices in $B$ that have at least $t$ neighbors in $S$ by $R_t(S)$. Similarly, we have the definition of $R_t(S)$ with $S\subset B$. In \cite{lund}, Lund and Saraf proved the following theorem. 
\begin{theorem}\label{lm3}
If a set $S$ of vertices in $B$ such that 
$|S|\ge 2(t-1)|B|/\deg(A)$, then $|R_t(S)|\ge c|A|$, where $c=(t-1)/(t-1+2\deg(A)\mu^2)$, $\mu=\lambda_3/\lambda_1$.
\end{theorem}

\section{Plane-incidence graphs}

We now construct the plane-incidence graph $G_P=(A,B,E)$ as follows. The first vertex part is the set of all $k$-planes, and the second vertex part is the set of all $h$-planes. There is an edge between a $k$-plane $v$ and a $h$-plane $p$ if $v$ lies on $p$.
It is easy to check that 
\[|A|=\frac{q^d(q^d-1)\cdots (q^d-q^{k-1})}{q^k(q^k-1)\cdots (q^k-q^{k-1})},~ |B|=\frac{q^d(q^d-1)\cdots (q^d-q^{h-1})}{q^h(q^h-1)\cdots (q^h-q^{h-1})}.\]
Now, we will count the degree of each vertex of the graph $G_P$. We first need the following lemmas. 
\begin{lemma}\label{lm11}
Let $U=\spn\{u_1,\ldots,u_k\}+u_{k+1}$ and $V=\spn\{v_1,\ldots,v_k\}+v_{k+1}$ be two $k$-planes in $\mathbb{F}_q^d$. Then $U\equiv V$ if and only if $\spn\{u_1,\ldots,u_k\}\equiv \spn\{v_1,\ldots,v_k\}$ and $u_{k+1}\in V, v_{k+1}\in U$.
\end{lemma}
\begin{proof}
If $\spn\{u_1,\ldots,u_k\}\equiv\spn\{v_1,\ldots,v_k\}$ and $u_{k+1}\in V, v_{k+1}\in U$, then it is easy to check that $U\equiv V$. For the inverse case, if $U\equiv V$, then $u_{k+1}\in V$ and $v_{k+1}\in U$. We need to prove that $\spn\{u_1,\ldots,u_k\}\equiv\spn\{v_1,\ldots,v_k\}$. Indeed, without loss of generality, we assume that there exists an element $u_i$ for some $1\le i\le k$ such that $u_i\not\in \spn\{v_1,\ldots,v_k\}$, then we will prove that this leads to a contradiction. Since $U\equiv V$, $u_i+u_{k+1}\in V$. Therefore, there eixst elements $a_1,\ldots,a_k\in \mathbb{F}_q$ such that $u_i+u_{k+1}=\sum_{i=1}^ka_iv_i+v_{k+1}.$ On the other hand, since $u_{k+1}\in V$, there exist elements $b_1,\ldots,b_k\in \mathbb{F}_q$ such that $u_{k+1}=\sum_{i=1}^kb_iv_i+v_{k+1}.$ This implies that 
$u_i=\sum_{i=1}^k(a_i-b_i)v_i,$
which leads to a contradiction. \end{proof}

\begin{lemma}\label{lm22}
Let $U=\spn\{u_1,\ldots,u_k\}+u_{k+1}$, $V=\spn\{v_1,\ldots,v_k\}+v_{k+1}$ be two $k$-planes in $\mathbb{F}_q^d$. For any $h>k$, if the $h$-plane $H=\spn\{t_1,\ldots,t_h\}+t_{h+1}$ contains both of them then $H$ can be written as $H=\spn\{t_1,\ldots,t_h\}+u_{k+1}$ and  $u_1,\ldots,u_k, v_1,\ldots,v_k,u_{k+1}-v_{k+1}\in \spn\{t_1,\ldots,t_h\}$.
\end{lemma}
\begin{proof}
First we need to prove that for any vector $x\in H$, $H$ can be written as $H=\spn\{t_1,\ldots,t_h\}+x$. Indeed, since $x\in H$, $x$ can be presented as $x=\sum_{i=1}^ha_it_i+t_{h+1}$ with $a_i\in \mathbb{F}_q$. Let $y=\sum_{i=1}^hb_it_i+t_{h+1}$ be a vector in $ H$, then $y$ can also be written as $y=\sum_{i=1}^h(b_i-a_i)t_i+x$. This implies that $y\in \spn\{t_1,\ldots,t_h\}+x$. The inverse case $\spn\{t_1,\ldots,t_h\}+x\subset H$ is trivial.

If $H$ contains both $U$ and $V$, then $H$ can be presented as $H=\spn\{t_1,\ldots,t_h\}+u_{k+1}$ since $u_{k+1}\in H$. It is easy to see that $u_i\in \spn\{t_1,\ldots,t_h\}$ for all $1\le i\le k$. Since $V$ is contained in $H$, $v_{k+1}\in H$, which implies that $v_{k+1}-u_{k+1}\in \spn\{t_1,\ldots,t_h\}$, and $v_i\in \spn\{t_1,\ldots,t_h\}$ for all $1\le i\le k$.
\end{proof}

Using Lemma \ref{lm11} and Lemma \ref{lm22}, we obtain that the degree of each $h$-plane is 
\[\frac{q^h}{q^k}\prod_{i=0}^{k-1}\frac{q^h-q^i}{q^k-q^i}=(1+o(1))q^{(h-k)(k+1)}.\]
In order to count the degree of each $k$-plane, we will use similar arguments as in the proof of \cite[Theorem 2.3]{iom}. Let $x(h,k)$ be the numer of distinct $k$-planes in a $h$-plane. Let $y(h,k)$ be the number of distinct $h$-planes in $\mathbb{F}_q^d$ containing a fixed $k$-plane. Then we have 
\[y(h,k)=\frac{x(h,k)x(d,h)}{x(d,k)}.\]
On the other hand, we just proved that
\[x(h,k)=\frac{q^h}{q^k}\prod_{i=0}^{k-1}\frac{q^h-q^i}{q^k-q^i},\]
which implies that
\[y(h,k)=\prod_{i=k}^{h-1}\frac{q^{d-i}-1}{q^{h-i}-1}=(1+o(1))q^{(d-h)(h-k)}.\]
In short, the degree of each $k$-plane is $(1+o(1))q^{(d-h)(h-k)}$.
We are now ready to bound the third eigenvalue of $M$ in the following lemma.
\begin{lemma}\label{eigen}
The third eigenvalue of $M$ is bounded by $q^{\left((d-h)h+k(2h-d-k+1)\right)/2}$.
\end{lemma}
\begin{proof}
Let $M$ be the adjacency matrix of $G_P$, which has the form 
$$M=\begin{bmatrix}
0&N\\
N^T&0
\end{bmatrix},
$$
where $N$ is a $q^{(d-k)(k+1)}\times q^{(d-h)(h+1)}$ matrix, and $N_{vp}=1$ if $v\in p$, and zero otherwise. Therefore, 
$$M^2=\begin{bmatrix}
NN^T&0\\
0&N^TN
\end{bmatrix}.
$$
It follows from Lemma \ref{lm1} that it suffices to bound the second eigenvalue of $NN^T$. Given any two $k$-planes $V_1=\spn\{u_1,\ldots,u_k\}+u_{k+1}$ and $V_2=\spn\{v_1,\ldots,v_k\}+v_{k+1}$, we now count the number of their common neighbors, i.e. the number of $h$-planes containing both of them.  We assume that $H=\spn\{t_1,\ldots,t_h\}+t_{h+1}$ is a $h$-plane supporting $V_1$ and $V_2$. Then it follows from Lemma \ref{lm11} and Lemma \ref{lm22} that $H$ can be written as $H=\spn\{t_1,\ldots,t_h\}+u_{k+1}$ and $u_1,\ldots,u_k,v_1,\ldots,v_k, u_{k+1}-v_{k+1}\in \spn\{t_1,\ldots,t_h\}$. Thus the number of $h$-planes supporting $V_1$ and $V_2$ depends on the rank of the following system of vectors $\rank(V_1,V_2):=\{u_1,\ldots,u_k,v_1,\ldots,v_k,v_{k+1}-u_{k+1}\}$. We also note that $k+1\le \rank (V_1,V_2)\le 2k+1$ since $V_1$ and $V_2$ are distinct. We assume that $\rank(V_1,V_2)=t$ and $\spn\{u_1,\ldots,u_k,v_1,\ldots,v_k,v_{k+1}-u_{k+1}\}\equiv \spn\{w_1,\ldots,w_t\}$ with $w_i\in \mathbb{F}_q^d$ for $1\le i\le t$, then the number of $h$-planes containing both $V_1$ and $V_2$ equals the number of $(h-t)$-tuples of vectors $\{x_1,\ldots,x_{h-t}\}$ in $\mathbb{F}_q^d$ such that $\rank\{w_1,\ldots,w_t,x_1,\ldots,x_{h-t}\}=h$. Thus the number of common neighbors of $V_1$ and $V_2$ is
\[\frac{(q^d-q^t)(q^d-q^{t+1})\ldots (q^d-q^{h-1})}{(q^h-q^t)\ldots (q^h-q^{h-1})}=(1+o(1))q^{(d-h)(h-t)}.\]
Therefore, $NN^T$ can be presented as
\begin{eqnarray}\label{eq11}
NN^T&=&q^{(d-h)(h-2k-1)}J+\left(q^{(d-h)(h-k)}-q^{(d-h)(h-2k-1)}\right)I\\
&+&\sum_{k+1\le t\le 2k}\left(q^{(d-h)(h-t)}-q^{(d-h)(h-2k-1)}\right)E_{t}\nonumber
\end{eqnarray}
where $I$ is the identity matrix and $J$ is the all-one matrix, and for each $t$, $E_t$ are the adjacency matrix of the graphs $G(E_t)$: $V(G(E_t))$ is the set of all $k$-planes, and there is an edge between two $k$-planes $V_1=\spn\{u_1,\ldots,u_k\}+u_{k+1}$ and $V_2=\spn\{v_1,\ldots,v_k\}+v_{k+1}$ if and only if $\rank(V_1,V_2)=t$. We note that these graphs are regular, and we count their degree as follows. For each $k+1\le t\le 2k$, and each vertex $V_1$, we now count the number of $k$-planes $V_2$ such that $\rank(V_1,V_2)=t$. In order to that, we consider two following cases
\begin{enumerate}
\item If $u_{k+1}-v_{k+1}\in \spn\{u_1,\ldots,u_k,v_1,\ldots,v_k\}$, then the number of $V_2$ is 
\[\frac{k!q^t(q^d-q^k)\cdots(q^d-q^{t-1})(q^t-q^{t-k})\cdots(q^t-q^{k-1})}{(t-k)!(2k-t)!q^k(q^k-1)\cdots(q^k-q^{k-1})}=(1+o(1))q^{(t-k)(d-t+k+1)},\]
where $(q^d-q^k)\cdots(q^d-q^{t-1})/(t-k)!$ is the number of $(t-k)$-tuples $\{v_1,\ldots,v_{t-k}\}$ such that $\rank\{u_1,\ldots,u_k,v_1,\ldots,v_{t-k}\}=t$, and $(q^t-q^{t-k})\cdots(q^t-q^{k-1})/(2k-t)!$ is the number of $(2k-t)$-tuples of vectors $\{v_{t-k+1},\ldots,v_k\}$ in $\spn\{u_1,\ldots,u_k,v_1,\ldots,v_{t-k}\}$ such that $\rank\{v_1,\ldots,v_k\}=k$, and $q^t$ is the number of choices of $u_{k-1}-v_{k+1}$ in $\spn\{u_1,\ldots,u_k,v_1,\ldots,v_k\}$, the term $q^k(q^k-1)\cdots(q^k-q^{k-1})/k!$ is the number of different ways to present a $k$-plane. We note that for each choice of $u_{k-1}-v_{k-1}$, then $v_{k+1}$ is determined uniquely. 
\item If $u_{k+1}-v_{k+1}\not\in \spn\{u_1,\ldots,u_k,v_1,\ldots,v_k\}$, then the number of $V_2$ is 
\[\frac{k!(q^d-q^k)\cdots(q^d-q^{t-1})(q^{t-1}-q^{t-k-1})\cdots (q^{t-1}-q^{k-1})}{(t-k)!(2k-t+1)!q^k(q^k-1)\cdots(q^k-q^{k-1})}=(1+o(1))q^{(t-k)(d-t+k+2)-1-k},\]
where $(q^d-q^k)\cdots(q^d-q^{t-1})/(t-k)!$ is the number of $(t-k)$-tuples $\{v_1,\ldots,v_{t-k-1},u_{k+1}-v_{k+1}\}$ such that $\rank\{u_1,\ldots,u_k,v_1,\ldots,v_{t-k-1},u_{k+1}-v_{k+1}\}=t$, and the second term $(q^{t-1}-q^{t-k-1})\cdots(q^{t-1}-q^{k-1})/(2k-t+1)!$ is the number of $(2k-t+1)$-tuples $\{v_{t-k},\ldots,v_k\}$ in $\spn\{u_1,\ldots,u_k,v_1,\ldots,v_{t-k-1}\}$ such that $\rank\{v_1,\ldots,v_k\}=k$.
\end{enumerate}
Therefore, for each $t$, the degree of any vertex in $V(G(E_t))$ is
\[(1+o(1))\left(q^{(t-k)(d-t+k+1)}+q^{(t-k)(d-t+k+2)-1-k}\right)=(1+o(1))q^{(t-k)(d-t+k+1)}.\] 
Let $v_3=(v_1,\ldots,v_{|A|},u_1,\ldots,u_{|B|})$ be the third eigenvector of $M^2$. Lemma \ref{lm1} implies that $(v_1,\ldots,v_{|A|})$ is an eigenvector of $NN^T$ corresponding to the eigenvalue $\lambda_3^2$. It follows from the equation (\ref{eq11}) that
\[\left(\lambda_3^2-(q^{(d-h)(h-k)}-q^{(d-h)(h-2k-1)})\right)v_3=\left(\sum_{k+1\le t\le 2k}\left(q^{(d-h)(h-t)}-q^{(d-h)(h-2k-1)}\right)E_t\right)v_3.\]
Hence, $v_3$ is an eigenvector of \[\sum_{k+1\le t\le 2k}\left(q^{(d-h)(h-t)}-q^{(d-h)(h-2k-1)}\right)E_t.\]Since eigenvalues of sum of matrices are bounded by sum of largest eigenvalues
of the summands, we have 
\[\lambda_3^2\le q^{(d-h)(h-k)}+q^{(d-h)(h-2k-1)}+kq^{(d-h)h+k(2h-d-k+1)}\le (2k+1)q^{(d-h)h+k(2h-d-k+1)},\]
which completes the proof of lemma. \end{proof}
\section{Proofs of Theorems \ref{thm1}, \ref{thm2}, and \ref{thm3}}
\begin{proof}[Proof of Theorem \ref{thm1}]
Since the degree of each $k$-plane is $(1+o(1))q^{(d-h)(h-k)}$, and the number of $h$-planes is $(1+o(1))q^{(d-h)(h+1)}$, we have 
\[\frac{\deg(A)}{|B|}=(1+o(1))q^{-(d-h)(k+1)}.\]
Thus, Theorem \ref{thm1} follows by combining Lemma \ref{expander} and Lemma \ref{eigen}.\end{proof}
\begin{proof}[Proof of Theorems \ref{thm2} and \ref{thm3}]
Combining Theorem \ref{lm3} and Lemma \ref{eigen}, Theorem \ref{thm2} and Theorem \ref{thm3} follow.
\end{proof}

\end{document}